\documentclass[envcountsame]{llncs}
\usepackage{amsmath,amssymb}
\usepackage{stmaryrd}
\usepackage[hyperindex,breaklinks,colorlinks=true,linkcolor=black,anchorcolor=black,citecolor=black,filecolor=black,menucolor=black,runcolor=black,urlcolor=black,pagebackref]{hyperref}
\usepackage{tikz}

\title{The Vitali Covering Theorem\\ in the Weihrauch Lattice\thanks{This article is dedicated to Rod Downey on the occasion of his sixtieth birthday.}}

\author{Vasco Brattka\inst{1,2}\thanks{Vasco Brattka is supported by the National Research Foundation of South Africa.}
            \and Guido Gherardi\inst{3}
            \and Rupert H\"olzl\inst{2}\thanks{Rupert H\"olzl was partly supported by the Ministry of Education of Singapore through grant R146-000-184-112 (MOE2013-T2-1-062).}
            \and Arno Pauly\inst{4}}
\institute{Dept.\ of Mathematics \& App.\ Maths., University of Cape Town, South Africa
             \and Faculty of Computer Science, Universit\"at der Bundeswehr M\"unchen, Germany
             \and Dipartimento di Filosofia e Comunicazione, Universit\`{a} di Bologna, Italy
             \and D{\'e}partment d'Informatique, Universit{\'e} libre de Bruxelles, Belgium
           \email{Vasco.Brattka@cca-net.de}, 
           \email{Guido.Gherardi@unibo.it}, 
           \email{r@hoelzl.fr},
           \email{Arno.Pauly@cl.cam.ac.uk} }

\usepackage{xcolor}	
	\usepackage{soul}

\makeatletter
\DeclareRobustCommand{\rvdots}{%
  \vbox{
    \baselineskip4\p@\lineskiplimit\z@
    \kern-\p@
    \hbox{.}\hbox{.}\hbox{.}
  }}
\makeatother

\def\N{\mathbb{N}}

\def\C{\mathrm{C}}

\def\PC{\mathrm{PC}}

\def\WKL{\mathrm{WKL}}
\def\WWKL{\mathrm{WWKL}}

\newcommand{\uint}{{[0,1]}}

\def\AA{{\mathcal A}}

\def\II{{\mathcal I}}
\def\JJ{{\mathcal J}}

\def\SS{{\mathcal S}}

\def\IN{{\mathbb{N}}}
\def\IQ{{\mathbb{Q}}}
\def\IR{{\mathbb{R}}}

\def\In{\subseteq}

\def\mto{\rightrightarrows}

\def\dom{{\rm dom}}
\def\range{{\rm range}}

\def\diam{{\rm diam}}

\def\mod{{\rm mod}}

\def\Int{{\rm Int}}

\newcommand{\SO}[1]{{{\bf\Sigma}^0_{#1}}}

\def\WKL{\text{\rm\sffamily WKL}}
\def\WWKL{\text{\rm\sffamily WWKL}}

\def\RCA{\text{\rm\sffamily RCA}}

\def\P{\mbox{\rm\sffamily P}}
\def\PC{\mbox{\rm\sffamily PC}}
\def\C{\mbox{\rm\sffamily C}}

\def\VCT{\text{\rm\sffamily VCT}}

\def\ACT{\text{\rm\sffamily{ACT}}}

\usepackage{nicefrac}
\usepackage{mathtools}

\def\leqW{\mathop{\leq_{\mathsf{W}}}}
\def\equivW{\mathop{\equiv_{\mathsf{W}}}}

\def\leqSW{\mathop{\leq_{\mathsf{sW}}}}

\def\equivSW{\mathop{\equiv_{\mathsf{sW}}}}

\newcommand{\dash}{\mbox{-}}

\date{\today}

\newtheorem{fact}[theorem]{Fact}

\begin{document}

\maketitle

\begin{abstract}
We study the uniform computational content of the Vitali Covering Theorem for intervals using the tool of Weihrauch reducibility. 
We show that a more detailed picture emerges than what a related study by Giusto, Brown, and Simpson has revealed in the setting of reverse mathematics. 
In particular, different formulations of the Vitali Covering Theorem turn out to have different uniform computational content.
These versions are either computable or closely related to uniform variants of Weak Weak K\H{o}nig's Lemma.
\end{abstract}

\section{Introduction}

In order to analyze the uniform computational content of the Vitali Covering Theorem in different versions
it is useful to introduce some terminology that will allow us to phrase these  versions
in succinct terms. 

Let $\II=(I_n)_n$ be a sequence of open intervals $I_n\In\IR$, let $x\in\IR$ and $A\In\IR$.
We say that $x\in\IR$ is {\em captured} by $\II$, if for every $\varepsilon>0$ there exists
some $n\in\IN$ with $\diam(I_n)<\varepsilon$ and $x\in I_n$.
We call $\II$ a {\em Vitali cover} of $A$, if every $x\in A$ is captured by $\II$.
We say that $\II$ is {\em saturated}, if $\II$ is a Vitali cover of $\bigcup\II:=\bigcup_{n=0}^\infty I_n$.
Finally, we say that $\II$ {\em eliminates} $A$,
if the $I_n$ are pairwise disjoint and $\lambda(A\setminus\bigcup\II)=0$, where $\lambda$ denotes the Lebesgue measure on $\IR$.

Using this terminology we can now formulate the Vitali Covering Theorem 
(see Richardson~\cite[Theorem~7.3.2]{Ric09}).

\begin{theorem}[Vitali Covering Theorem]
	\label{thm:Vitali}
	Let $A\In[0,1]$ be Lebesgue measurable and let $\II$ be a sequence of intervals.
	If $\II$ is a Vitali cover of $A$, then there exists a subsequence $\JJ$ of $\II$
	that eliminates $A$.
\end{theorem}

The Vitali Covering Theorem has been studied in reverse mathematics by Brown, Giusto, and Simpson~\cite{BGS02} and was shown to coincide in proof strength with the  well-known principle $\WWKL_0$ that stands for 
Weak Weak K\H{o}nig's Lemma, see Simpson~\cite{Sim09}.
The following result can be found in Brown, Giusto, and Simpson~\cite[Theorems~3.3 and 5.5]{BGS02} and also in Simpson~\cite[Theorems~X.1.9 and X.1.13]{Sim09}.
For a related study in constructive analysis, see Diener and Hedin~\cite{DH12}.

\begin{theorem}[Brown, Giusto, and Simpson~\cite{BGS02}]
\label{thm:BGS02}
Over $\RCA_0$, the following statements are equivalent to each other:
\begin{enumerate}
\item Weak Weak K\H{o}nig's Lemma $\WWKL_0$,
\item The Vitali Covering Theorem (Theorem~\ref{thm:Vitali}) for $A=[0,1]$,
\item For any sequence of intervals $\II=(I_n)_n$ with $[0,1]\In\bigcup\II$ it holds that $\sum_{n=0}^\infty\lambda(I_n)\geq1$.
\end{enumerate}
\end{theorem}

In a series of articles \cite{GM09,Pau10,Pau10a,BG11,BG11a,BBP12,BGM12,BGH15a,DDH+16} by different authors the Weihrauch lattice was established as a 
uniform, resource-sensitive and hence more fine-grained version of reverse mathematics.
Starting with work of Brattka and Pauly~\cite{BP10}, Dorais, Dzhafarov, Hirst, Mileti and Shafer~\cite{DDH+16} and Brattka, Gherardi and H\"olzl \cite{BGH15,BGH15a},
probabilistic problems were studied in the Weihrauch lattice. In particular {\em positive choice} $\PC_X$ was considered, which is the problem of finding
a point in a closed $A\In X$ of positive measure, and the following relation to Weak Weak K\H{o}nig's Lemma was established in the
Weihrauch lattice \cite[Proposition~8.2 and Theorem~9.3 and its proof]{BGH15a}.

\begin{fact}[Weak Weak K\H{o}nig's Lemma]
\label{fact:WWKL} 
\
\begin{enumerate}
\item $\WWKL\equivSW\PC_{2^\IN}\equivSW\PC_{[0,1]}$,
\item $\WWKL\times\C_\IN\equivSW\PC_{\IN\times2^\IN}\equivSW\PC_{\IR}$.
\end{enumerate}
\end{fact}

Here $\equivSW$ stands for equivalence with respect to strong Weihrauch reducibility. 
We will provide exact definitions of the relevant terms in the following Section~\ref{sec:preliminaries}.
In this article we are going to extend the work by Brown, Giusto and Simpson~\cite{BGS02}
using the tools of the Weihrauch lattice and we will demonstrate how the above mentioned equivalence classes
and others feature in this approach. 

One of our main insights is related to the observation that different logical formulations of the Vitali Covering Theorem
turn out to have different uniform computational content, a phenomenon that appeared in a similar way in the study of the Baire Category Theorem by
Brattka and Gherardi~\cite{BG11a} and Brattka, Hendtlass and Kreuzer~\cite{BHK15a}.
The following three propositional formulas essentially correspond to the different logical formulations of the Vital Covering Theorem that we consider:

\begin{enumerate}
\setcounter{enumi}{-1}
	\item $(S\wedge C)\to E$,
	\item $(S\wedge\neg E)\to\neg C$,
	\item $\neg E\to(\neg S\vee\neg C)$.
\end{enumerate}

Here $S$ corresponds to the statement that the input sequence is saturated, $C$ to the statement that it is a cover
and $E$ to the statement that there is an eliminating subsequence. 
The stated propositional formulas are equivalent to each other when we have the full strength of classical logic at our disposal.
More precisely, we are going to use the following versions of the Vitali Covering Theorem for the special case $A=[0,1]$:

\begin{enumerate}
\setcounter{enumi}{-1}
	\item $\VCT_0$: For every Vitali cover $\II$ of $[0,1]$ there exists a subsequence $\JJ$ of $\II$ that eliminates $[0,1]$.
	\item $\VCT_1$: For every saturated $\II$ that does not admit a subsequence which eliminates $[0,1]$, there exists a point $x\in[0,1]$ that is not covered by $\II$.
	\item $\VCT_2$: For every sequence $\II$ that does not admit a subsequence which eliminates~$[0,1]$, there exists a point $x\in[0,1]$ that is not captured by $\II$.
\end{enumerate}

It is clear that 0.\ is equivalent to 2.\ since they are contrapositive forms of each other.
We also obtain ``0.$\Rightarrow$1.'' since every saturated cover of $[0,1]$ is a Vitali cover of $[0,1]$. 
Finally, we obtain ``1.$\Rightarrow$0.'' since every Vitali cover $\II$ of $[0,1]$ can be extended to a saturated sequence $\II'$ by only adding intervals
that do not overlap with the closed set $[0,1]$. Every subsequence $\JJ'$ of $\II'$ that eliminates $[0,1]$ then leads to a subsequence $\JJ$ of $\II$
that eliminates $[0,1]$.
Our main results on the Vitali Covering Theorem can now be phrased as follows.
The proofs will be presented in Section~\ref{sec:Vitali}. 

\begin{theorem}[Vitali Covering Theorem]
We obtain that
\begin{enumerate}
\setcounter{enumi}{-1}
\item  $\VCT_0$ is computable, 
\item $\VCT_1\equivSW\PC_{[0,1]}\equivSW\WWKL$ and 
\item $\VCT_2\equivSW\PC_\IR\equivSW\WWKL\times\C_\IN$.
\end{enumerate}
\end{theorem}

It can be argued that $\C_\IN$ is the analogue of $\SO{1}$--induction in the Weihrauch lattice (see Brattka and Rakotoniaina~\cite{BR15}) and
hence the classes $\WWKL$ and $\WWKL\times\C_\IN$ have no distinguishable non-uniform content in reverse mathematics,
where $\SO{1}$--induction is already included in $\RCA_0$.

In this context it is also interesting to note that the equivalence classes of $\WWKL$ and $\WWKL\times\C_\IN$ characterize certain
natural classes of probabilistic problems. In \cite[Corollary~3.4]{BGH15a} the following was proved.

\begin{fact}[Las Vegas Computability]
\label{fact:Las-Vegas}
The following holds for any $f$.
\begin{enumerate}
\item $f\leqW\PC_{[0,1]}\iff f$ is Las Vegas computable,
\item $f\leqW\PC_\IR\iff f$ is Las Vegas computable with finitely many mind changes.
\end{enumerate}
\end{fact}

Since these classes of probabilistically computable maps will not play any further role in this article,
we will skip the precise definitions and refer the interested reader to Brattka, Gherardi and H\"olzl~\cite{BGH15,BGH15a}.

In Section~\ref{sec:additivity} we further analyze item 3.\ of Theorem~\ref{thm:BGS02}, a statement which is related
to countable additivity in reverse mathematics. We show that there is a formalization $\ACT$ of this statement
that we call {\em Additive Covering Theorem} and that turns out to be equivalent to $*\dash\WWKL$, yet
another variant of Weak Weak K\H{o}nig's Lemma that is even weaker than $\WWKL$ from the uniform perspective. 
In the diagram in Figure~\ref{fig:Vitali} we present a survey of our results.

\section{Preliminaries}
\label{sec:preliminaries}

We assume that the reader is familiar with the concepts defined in the introductory part of Brattka, Gherardi, and H\"olzl~\cite[Section 2]{BGH15a}. 
We recall some of the most central concepts.
Firstly, Weihrauch reducibility and its strong counterpart are defined for multi-valued functions $f:\In X\mto Y$ on represented
spaces $X,Y$. Representations are surjective partial mappings  from Baire space $\IN^\IN$ onto the represented spaces
and they provide the necessary structures to speak about computability and other concepts.
Since we are not using representations in any formal way here, we refrain from presenting further details and we point the reader to Weihrauch~\cite{Wei00} and Pauly~\cite{Pau16}.

\begin{definition}[Weihrauch reducibility]\rm
Let $f:\In X\mto Y$ and $g:\In W\mto Z$ be multi-valued functions on represented spaces. 
\begin{enumerate}
\item $f$ is said to be {\em Weihrauch reducible} to $g$, in symbols $f\leqW g$, if there are computable
	$K:\In X\mto W$, $H:\In X\times Z\mto Y$ with 
      $\emptyset\not=H(x, gK(x))\In f(x)$ for all $x\in\dom(f)$.
\item $f$ is said to be {\em strongly Weihrauch reducible} to $g$, in symbols $f\leqSW g$, if there are computable 
      $K:\In X\mto W$, $H:\In Z\mto Y$ with $\emptyset\not=HgK(x)\In f(x)$ for all $x\in\dom(f)$.
\end{enumerate}
The corresponding equivalences are denoted by $\equivW$ and $\equivSW$, respectively.
\end{definition}

In some results we are referring to products of multi-valued functions, which we define next.

\begin{definition}[Products]\rm
For $f:\In X\mto Y$ and $g:\In W\mto Z$ we define $f\times g:\In X\times W\mto Y\times Z$
by $(f\times g)(x,w):=f(x)\times g(w)$ and $\dom(f\times g):=\dom(f)\times\dom(g)$.
\end{definition}

Since we are going to prove that certain versions of the Vitali Covering Theorem can be characterized with the help of certain versions of positive choice,
we need to define positive choice next.  For this purpose we use
the negative information representation of $\AA_-(X)$, which represents closed sets $A\In X$ by enumerating rational open balls $B(x_i,r_i)$ that exhaust
the complement of $A$, that is $X\setminus A=\bigcup_{i=0}^\infty B(x_i,r_i)$.
The $x_i$ are taken from some canonical dense subset of $X$ and the $r_i$ are rational numbers.
For details we refer the reader to Brattka, Gherardi and H\"olzl~\cite{BGH15a}.

\begin{definition}[Choice and positive choice]\rm
Let $X$ be a separable metric space with a Borel measure $\mu$ and let $\AA_-(X)$ denote the set of closed subsets $A\In X$ with respect to negative information.
\begin{enumerate}
\item By $\C_X:\In\AA_-(X)\mto X,A\mapsto A$ we denote the {\em choice problem} of $X$
        with $\dom(\C_X):=\{A:A\not=\emptyset\}$.
\item By $\PC_X:\In\AA_-(X)\mto X,A\mapsto A$ we denote the {\em positive choice problem} of $X$
        with $\dom(\PC_X):=\{A:\mu(A)>0\}$.
\end{enumerate}
\end{definition}

We will mostly work with the real numbers $\IR$ or the unit interval $[0,1]$, both equipped with the Lebesgue measure $\lambda$.
In Section~\ref{sec:additivity} we will also use a quantitative version $\P_{>\varepsilon}\C_{[0,1]}$ of $\PC_{[0,1]}$ which is the restriction 
of $\PC_{[0,1]}$ to closed sets $A\In[0,1]$ with $\lambda(A)>\varepsilon$ for $\varepsilon>0$.

\section{Vitali Covering in the Weihrauch degrees}
\label{sec:Vitali}

We now translate the three logically equivalent versions of the Vitali Covering Theorem that were presented in the introduction 
into their corresponding multi-valued functions and hence into Weihrauch degrees. 

By $\Int$ we denote the set of sequences $(I_n)_n$ of open Intervals $I_n=(a,b)$ with~$a,b\in\IQ$ where we let $(a,b)=\emptyset$ if $b\leq a$. Formally we represent $\Int$ using the canonical representation of the set~$(\IQ^2)^\IN$.

\begin{definition}[Vitali Covering Theorem]\rm
	We define the following multi-valued functions.
	\begin{enumerate}
	\setcounter{enumi}{-1}
		\item
		$\VCT_0: \In\Int\mto\Int,\II\mapsto\{\JJ: \JJ\mbox{ is a subsequence of $\II$ that eliminates }[0,1]\}$
		and $\dom(\VCT_0)$ contains all $\II\in\Int$ that are Vitali covers of $[0,1]$.
		\item
		$\VCT_1:\In\Int\mto[0,1],\II\mapsto[0,1]\setminus\bigcup\II$
		and $\dom(\VCT_1)$ contains all $\II\in\Int$ that are saturated and that do not have a subsequence that eliminates $[0,1]$.
		\item
		$\VCT_2 :\In\Int\mto[0,1],\II\mapsto\left\{x\in[0,1]: \mbox{$x$ is not captured by $\II$}\right\}$ and\linebreak
		$\dom(\VCT_2)$ contains all $\II\in\Int$ that do not have a subsequence that  eliminates $[0,1]$.
	\end{enumerate}
\end{definition}

We note that $\dom(\VCT_1)\In\dom(\VCT_2)$ and that $\VCT_1$ is a restriction of $\VCT_2$ (see Proposition~\ref{prop:VCT1-VCT2}). 
By the Vitali Covering Theorem (Theorem~\ref{thm:Vitali}) the sequences $\II\in\dom(\VCT_2)$ cannot be Vitali covers of $[0,1]$.

\subsection{The computable version}

Brattka and Pauly~\cite{BP10} noticed that $\VCT_0$~is computable; we will give a formal proof in this subsection. As a preparation we need the following lemma, where for $A\In\IR$ we denote by $A^\circ$ and $\partial A$
the interior and the boundary of $A$, respectively.

\begin{lemma}
	\label{lem:VCT0}
	Let $A\In[0,1]$ be a closed set with $\lambda(A)>0$ and $\lambda(\partial A)=0$.
	If $\II=(I_n)_{n\in\N}$ is a Vitali cover of $A$, then the subsequence $\II_A$ of $\II$
	that consists only of those $I_n$ with $I_n\In A$ is a Vitali cover of $A^\circ$.
\end{lemma}
\begin{proof}
	We note that $\lambda(A)>0$ and $\lambda(\partial A)=0$ implies $\lambda(A^\circ)=\lambda(A\setminus\partial A)>0$.
	In particular, $A^\circ\not=\emptyset$ and the sequence $\II_A$ is well-defined.
	We claim that $\II_A$ is saturated.
	Let $x\in\bigcup\II_A$ and $\varepsilon>0$. Then there is an $n$ such that $x\in I_n\In A$.
	Since $\II$ is a Vitali cover of $A$, there is some $k$ such that $x\in I_k\In I_n$ and $\diam(I_k)<\varepsilon$.
	In particular, $I_k\In A$ and hence $I_k$ is a component of $\II_A$. Thus $\II_A$ is saturated.
	Similarly, it follows that 
      $\bigcup\II_A=A^\circ$. 
      Here the inclusion ``$\In$'' follows from the definition of $\II_A$ and
	we only need to prove ``$\supseteq$''. For every $x\in A^\circ$ there is some $\varepsilon>0$ with $(x-\varepsilon,x+\varepsilon)\In A$
	and since $\II$ is saturated there is some $k$ with $x\in I_k\In (x-\varepsilon,x+\varepsilon)$. Hence $I_k$ is part of $\II_A$
	and $x\in\bigcup\II_A$. This shows that $\bigcup\II_A=A^\circ$, and hence $\II_A$ is a Vitali cover of $A^\circ$.
	\qed
\end{proof}

Now we are prepared to prove that $\VCT_0$ is computable.

\begin{theorem}
	\label{thm:VCT0}
	$\VCT_0$ is computable.
\end{theorem}
\begin{proof}
	Given a Vitali cover $\II$ of $[0,1]$, we need to find a subsequence $\JJ$ of $\II$ that eliminates $[0,1]$.
	We will compute such a subsequence inductively. Initially, $\JJ$~is an empty sequence.
	We start with $A_0:=[0,1]$ and $\II_0:=\II$.
	We assume that in step $n$ of the computation the set $A_n$ is a non-empty finite union of closed rational intervals
	with $\lambda(A_n)>0$ and that $\II_n$ is a Vitali cover of the interior~$A_n^\circ$.
	The fact that $A_n$ is a non-empty finite union of rational intervals implies $\lambda(\partial A_n)=0$.
	Given a Vitali cover $\II_n$ of $A_n^\circ$ there exists a subsequence $\JJ_n$ of $\II_n$ that eliminates $A_n^\circ$ by the
	Vitali Covering Theorem  (Theorem~\ref{thm:Vitali}). Since the Lebesgue measure $\lambda$ is upper semi-computable
	on closed sets $A\In[0,1]$, by a systematic search
	we can find a $k_n\in\IN$ and a finite subsequence $(I_0,...,I_{k_n})$ of~$\II_n$ of pairwise disjoint intervals
	such that
	\[0<\lambda\left(A_n^\circ\setminus\bigcup_{i=0}^{k_n}I_i\right)<2^{-n}.\]
	We compute
	$A_{n+1}:=A_n\setminus\bigcup_{i=0}^{k_n}I_i$
	as a finite union of closed rational intervals and
	we add the intervals $I_0,...,I_{k_n}$ to the set $\JJ$.
	Since $\lambda(\partial A_n)=0$, we obtain that $0<\lambda(A_{n+1})<2^{-n}$.
	We now compute $\II_{n+1}:=(\II_n)_{A_{n+1}}$ (as defined in Lemma~\ref{lem:VCT0}).
	Then $\II_{n+1}$ is a Vitali cover of~$A_{n+1}^\circ$ by Lemma~\ref{lem:VCT0} and we can continue the inductive construction in step $n+1$.
	Altogether, this construction leads to a subsequence $\JJ$ of $\II$ of pairwise disjoint intervals $\JJ$ such that ${[0,1]\setminus\bigcup\JJ}=\bigcap_{n=0}^\infty A_n$.
	Since $\lambda(A_n)<2^{-n}$, it follows that $\lambda([0,1]\setminus\bigcup\JJ)=0$. Hence $\JJ$ eliminates $[0,1]$.
\qed
\end{proof}

\subsection{The first non-computable version}

In the previous subsection we observed that the most straight-forward way of formalizing the Vitali Covering Theorem in the Weihrauch degrees is computable. 
To obtain non-computability results, we need to look at contrapositive versions of the theorem. 
The idea here is that given a collection of intervals $\II$ that violates some of the requirements for
being a Vitali cover, we want to find an $x \in[0,1]$ witnessing this
violation. Again, there is more than one formalization for this idea, as
there are different ways and degrees of violating the
requirements.

It will turn out that these different formalizations produce mathematical tasks of different
computational strengths, that is, falling into different Weihrauch degrees. The first result in this direction that we will prove is that $\VCT_1$
is strongly equivalent to Weak Weak K\H{o}nig's Lemma. This corresponds to Theorem~\ref{thm:BGS02} by Brown, Giusto, and Simpson.

To show $\WWKL\leqSW\VCT_1$ we will use the following lemma that shows that we can computably refine any sequence of open intervals to a saturated one.

\begin{lemma}[Vitalization]
	\label{lem:vitalization}
	There exists a computable map $V: \Int\to\Int$ such that
	$\bigcup\II=\bigcup V(\II)$
	for all $\II\in\Int$
	and $\range(V)$ only consists of saturated sequences of intervals.
\end{lemma}
\begin{proof}
	Given $\II=(I_n)_n$ we systematically add to $\II$ all rational intervals ${I=(a,b)}$ 
	for which there is an $n\in\IN$ with $I\In I_n$. This leads in a computable way to a saturated sequence $\JJ$ with $\bigcup\II=\bigcup\JJ$.
\qed
\end{proof}

Now we are prepared to prove that $\VCT_1$ is strongly equivalent to $\PC_{[0,1]}$.

\begin{theorem}
	\label{thm:VCT1}
	$\VCT_1\equivSW\PC_{[0,1]}$.
\end{theorem}
\begin{proof}
	Given a sequence $\II$ of open intervals with ${A=[0,1]\setminus\bigcup\II}$ and $\lambda(A)>0$,
	 by Lemma~\ref{lem:vitalization} we can compute a saturated sequence $V(\II)$ with ${A=[0,1]\setminus\bigcup V(\II)}$.
	Since $\lambda(A)>0$, it is clear that $V(\II)$ does not have a subsequence that eliminates $[0,1]$.
	Hence $V(\II)\in\dom(\VCT_1)$ and $\VCT_1(V(\II))= A$, which implies $\PC_{[0,1]}\leqSW\VCT_1$.
	
	Now let $\II$ be a saturated sequence of intervals that does not have a subsequence that eliminates $[0,1]$.
	Clearly we can compute $A:=[0,1]\setminus\bigcup\II$. Since $\II$~is a Vitali cover of $\bigcup\II$,
	there is a subsequence $\JJ$ of $\II$ that eliminates $\bigcup\II$ by the Vitali Covering Theorem 
	(Theorem~\ref{thm:Vitali}).
	If $\lambda(A)=0$, then this subsequence $\JJ$ also  eliminates~$[0,1]$. This is not possible by assumption and hence $\lambda(A)>0$.
	Consequently, $\VCT_1(\II)=\PC_{[0,1]}(A)$, which proves $\VCT_1\leqSW\PC_{[0,1]}$.
\qed
\end{proof}

Since it is known that $\PC_{[0,1]}$ has computable inputs that do not admit computable outputs (see for example Brattka, Gherardi and H\"olzl~\cite[Theorem~12]{BGH15}), 
we obtain the following corollary as an immediate consequence (which also follows by Lemma~\ref{lem:vitalization} from a classical result of Kreisel and Lacombe~\cite[Th\'eor\`eme~VI]{KL57} on singular coverings, 
see also \cite[Theorem~4.28]{Wei00}).

\begin{corollary}[Diener and Hedin~{\cite[Theorem~9]{DH12}}]
There exists a computable Vitali cover $\JJ$ of the computable points in $[0,1]$ so that every subsequence $\II=(I_n)_n$
consisting of pairwise disjoint intervals
satisfies $\sum_{n=0}^\infty\lambda(I_n)<1$.
\end{corollary}

\subsection{The second non-computable version}

The previous result identifies the computational strength of $\VCT_1$ with that of the well-studied Weihrauch degree $\WWKL$. 
The natural next  question to ask is whether $\VCT_2$ is of different strength and, if yes, to determine that strength precisely. 
Both questions will be answered in this section. We begin with the following observation.

\begin{proposition}
	\label{prop:VCT1-VCT2}
	$\VCT_1\leqSW\VCT_2$.
\end{proposition}
\begin{proof}
      If 
      $\II$ is a saturated sequence of rational open intervals that contains no subsequence that eliminates $[0,1]$,
	then $\II$ does not cover $[0,1]$ by the Vitali Covering Theorem  (Theorem~\ref{thm:Vitali}) and every point $x\in[0,1]$ which is not captured by $\II$ is a point that is not covered by $\II$, that is, $x\in[0,1]\setminus\bigcup\II$.
	Hence $\VCT_1$ is a restriction of $\VCT_2$ and, in particular, $\VCT_1\leqSW\VCT_2$.
\qed
\end{proof}

On the other hand, $\VCT_2$ can be reduced to $\PC_\IR$,
as the next result shows. Within the proof we will use the following definition 
from Brown, Giusto, and Simpson~\cite{BGS02}.
A sequence $\II=(I_n)_n$ of intervals is an {\em almost Vitali cover} of a 
Lebesgue measurable set $A \subseteq [0,1]$ if for all $\varepsilon>0$ and
\[U_\varepsilon:=\bigcup\{I_n:n\in\IN\mbox{ and }\diam(I_n)<\varepsilon\}\]
it holds that $\lambda(A\setminus U_\varepsilon)=0$.
In fact, Brown, Giusto, and Simpson~\cite[Theorem~5.6]{BGS02} (see Simpson~\cite[Theorem~X.1.13]{Sim09})
proved the following strengthening of the Vitali Covering Theorem  (Theorem~\ref{thm:Vitali}): 
every almost Vitali cover $\II$ of $[0,1]$ admits a subsequence $\JJ$ that eliminates $[0,1]$. We use this result to obtain the following reduction.

\begin{proposition}
	\label{prop:VCT2-PCR}
	$\VCT_2\leqSW\PC_\IR$.
\end{proposition}
\begin{proof}
	Let $\II=(I_n)_{n\in\N}$ be a sequence of rational open intervals that does not contain a subsequence
	that eliminates $[0,1]$. By Brown, Giusto, and Simpson~\cite[Theorem~5.6]{BGS02} we obtain
	that $\II$ is not even an almost Vitali cover of $[0,1]$, that is, there exists some $n\in\IN$ such that 
	$\lambda([0,1]\setminus U_{2^{-n}})>0$, with $U_\varepsilon$ as defined above.
      We let $A_n:=[0,1]\setminus U_{2^{-n}}$ for all $n$. Clearly $A_n\In\VCT_2(\II)$ for all $n$.
	Now we compute  
      $A:=\bigcup_{n=0}^\infty (2n+A_n)$, 
       where $n+X:=\{n+x:x\in X\}$ for all $X\In\IN$.
	Then $\lambda(A)>0$ and $\PC_\IR(A)$ yields a point $x$ with $(x\;\mod\; 2)\in\VCT_2(\II)$.
	This proves $\VCT_2\leqSW\PC_\IR$.
\qed
\end{proof}

Now we prove by a direct construction that $\VCT_2$ can compute itself concurrently with $\C_\IN$.

\begin{proposition}
\label{prop:CN-VCT2}
$\C_\N \times \VCT_2 \leqSW \VCT_2$.
\end{proposition}
\begin{proof}
For the purposes of this proof we treat sequences of intervals $\II=(I_n)_n$ as sets $\II=\{I_n:n\in\IN\}$ of intervals.
All sets of intervals that we are going to use can be enumerated in a natural way. 
 
Let $A$ be an instance of $\C_\mathbb{N}$ and $\II$ an instance of  $\VCT_2$, that is $\II$ does not have
a subsequence that eliminates $[0,1]$.
By $\mathcal{I}_{[a,b]}$ we denote the image of $\mathcal{I}$ under rescaling $\uint$ to $[a,b]$.\footnote{There is a slight ambiguity here, as we need to deal with open sets ranging beyond~$\uint$. 
We shall understand these to be \emph{small enough} in the sense that we cut away everything from a certain distance $\varepsilon_n$ on.
The exact constraints that these values $\varepsilon_n$ need to satisfy are given in the proof.} By $\mathcal{S}_{(a,b)}$ we denote some saturated and computably enumerable set of intervals with $\bigcup \mathcal{S}_{(a,b)} = (a,b)$, which exists by Lemma~\ref{lem:vitalization}. 

We use points of the form $x_n:=1-\frac{1}{n}$ for $n>1$ to subdivide the unit interval $[0,1]$ into countably many regions. 
In each of these regions with $n>1$ we will place countably many scaled copies of $\II$ into certain intervals of the form $[a_n,b_n]:=[x_n+2^{-n-1}, x_n+2^{-n}]$ 
and $[a_{n,j},b_{n,j}]:=[x_n - 2^{-2j},x_n - 2^{-2j-1}]$ for $j>n$.
We construct an instance $\JJ := \JJ_{\mathrm{P}} \cup \JJ_\II \cup \JJ_{\mathrm{R}} \cup \JJ_A$
 of $\VCT_2$ in four parts:

\begin{eqnarray*}
\JJ_{\mathrm{P}} &:=& \{(x_n - 2^{-j}, x_n + 2^{-j}) : n > 1, j > n\}\cup\{(x_n, 1+2^{-n}) : n > 1\}\\[0.2cm]
\JJ_{\II} &:=& \bigcup\limits_{n > 1\atop j > n} \mathcal{I}_{[a_{n,j} , b_{n,j}]}\cup\bigcup\limits_{n > 1\atop} \mathcal{I}_{[a_n, b_n]}\\
\JJ_{\mathrm{R}} &:=& \mathcal{S}_{(-2^{-1},a_{2,3})} \cup \bigcup\limits_{n > 1\atop j > n} \mathcal{S}_{(b_{n,j}, a_{n,j+1})}  
\cup   \bigcup\limits_{n > 1\atop} (\mathcal{S}_{(x_n, a_n)}
     \cup  \mathcal{S}_{(b_{n}, a_{n+1,n+2})})\\
\JJ_{A} &:=& \bigcup\limits_{n > 1\atop n-2 \notin A}\mathcal{S}_{(a_n - \varepsilon_n, b_n + \varepsilon_n)}    
      \cup\bigcup\limits_{{n > 1\atop j > n}\atop j-n-1 \notin A}\mathcal{S}_{(a_{n,j} - \varepsilon_j, b_{n,j} + \varepsilon_j)}
\end{eqnarray*}

\begin{figure}[tb]
\begin{tikzpicture}[scale=0.6,auto=left,every node/.style={fill=black!0}]

\def\rvdots{\raisebox{1mm}[\height][\depth]{$\huge\vdots$}};

 \draw [-] (0,7) -- (20,7); 
 \draw (0,6.9) -- (0,7.1)   node [above] {$0$};
 \draw (20,6.9) -- (20,7.1)   node [above] {$1$};
 \foreach \n in {2,3,4,5}{\draw (20-20/\n,6.9) -- (20-20/\n,7.1)   node [above] {$x_\n$};}
 \node at (17,7.4) {$...$};
 \node at (18,7.4) {$x_n$};
 \node at (19,7.4) {$...$};
 
 \draw [-,line width=0.5mm,dashed] (17.5,7) -- (18.5,7); 

 \draw [-,dashed] (17.5,7) -- (0,4);
 \draw [-,dashed] (18.5,7) -- (20,4);
 \draw [-,thin] (0,4) -- (20,4); 
 \draw (10,3.9) -- (10,4.1)   node [above] {$x_n$};
 \draw[gray] (12,3.9) -- (12,4.1)   node [above] {$x_n+2^j$};
  \draw[gray] (8,3.9) -- (8,4.1)   node [above] {$x_n-2^j$};
  
 \draw [-,line width=0.5mm] (14,4) -- (17,4); 
 \draw (14,3.9) -- (14,4.1)   node [below=5pt] {$a_n$};
 \draw (17,3.9) -- (17,4.1)   node [below=3pt] {$b_n$};

 \draw [-,line width=0.5mm] (6.5,4) -- (8,4); 
 \draw (6.48,3.9) -- (6.48,4.1)   node [below=5pt] {$a_{n,j+1}$};
 \draw (8.02,3.9) -- (8.02,4.1)   node [below=3pt] {$b_{n,j+1}$};

\draw [-,line width=0.5mm] (2,4) -- (5,4); 
 \draw (2,3.9) -- (2,4.1)   node [below=5pt] {$a_{n,j}$};
 \draw (5,3.9) -- (5,4.1)   node [below=3pt] {$b_{n,j}$};

 \node at (9.3,3.4) {$...$};


\end{tikzpicture}
  
\caption{Illustration of the intervals $[a_{n,j},b_{n,j}]$ and $[a_n,b_n]$ in correct order, but oversized.}
\label{fig:intervals}
\end{figure}
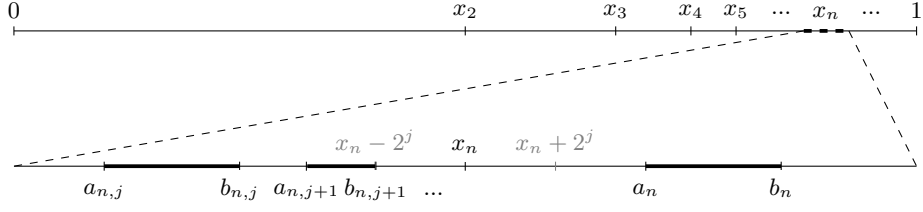

\noindent
Here $(\varepsilon_n)_n$ is a computable sequence of positive rational numbers that are subject to the following constraints
for all $n>1$ and $j>n$:
\[x_n<a_n-\varepsilon_n,\; b_n+\varepsilon_n<a_{n+1,n+2}-\varepsilon_{n+2}\mbox{ and }
b_{n,j}+\varepsilon_j<a_{n,j+1}-\varepsilon_{j+1}.\]
In Figure~\ref{fig:intervals} the construction is visualized.
Intuitively, we capture the point $1$ and all points $x_n=1-\frac{1}{n}$ using $\JJ_{\mathrm{P}}$. 
Using $\JJ_\II$ we place scaled copies of $\II$ into the intervals $[a_n,b_n]$ and $[a_{n,j},b_{n,j}]$ for $n>1$ and $j>n$.
The remainder of the unit interval is captured using $\JJ_{\mathrm{R}}$. 
Finally, those regions not corresponding to an index from $A$ are rendered invalid responses by capturing them using $\JJ_A$,
where the constraints on $\varepsilon_n$ above guarantee that no neighbor regions are touched.

Any point not captured by $\JJ$ must lie in one of the regions designated in the definition of $\JJ_\II$, and, as these are separated, 
we can compute the parameters of the region (thus producing the answer for the instance $A$ of $\C_\mathbb{N}$), 
and then scale the point back up to produce the answer to the instance $\II$ of $\VCT_2$.

It remains to prove that $\JJ$ actually is a valid input to $\VCT_2$, that is, that no collection $\SS\In\JJ$ of disjoint
intervals eliminates $\uint$.
Let $\mathcal{S} \subseteq \JJ$ be a disjoint collection of intervals. We distinguish two cases:

\smallskip

\noindent {\bf Case 1}: $(\exists n)\; (x_n, 1+2^{-n}) \in \mathcal{S}$. 
Then no set of the form ${(x_n - 2^{-j}, x_n + 2^{-j})}$ can be in $\mathcal{S}$. 
Choose $j$ such that $j - n - 1 \in A$. 
We claim that $\mathcal{S}$ cannot eliminate $[a_{n,j}, b_{n,j}]$: We have already seen that under the given conditions, 
we have for every $U \in \mathcal{S} \cap \JJ_{\mathrm{P}}$ that
$U \cap [a_{n,j} , b_{n,j}] = \emptyset$.
The same is true for ${U \in \mathcal{S} \cap (\JJ_{\mathrm{R}} \cup \JJ_A)}$ by construction and because $j - n - 1 \in A$. 
Thus, the only sets which could contribute to eliminating the interval $[a_{n,j},b_{n,j}]$ come from $\JJ_\II$, and more specifically, 
$\II_{[a_{n,j} , b_{n,j}]}$; 
but if these sets would eliminate $[a_{n,j} , b_{n,j}]$, then $\II$ would eliminate $\uint$, which is impossible.

\smallskip

\noindent {\bf Case 2}: $(\forall n)\; (x_n, 1+2^{-n}) \notin \mathcal{S}$. 
Let $n$ be such that $n - 2 \in A$. We claim that $\mathcal{S}$~cannot eliminate 
$[a_n, b_n]$. 
For $U \in \mathcal{S} \cap \JJ_{\mathrm{P}}$ we have that $U \cap [a_n, b_n] = \emptyset$ because 
$(x_n-2^{-j},x_n+2^{-j})\cap[a_n,b_n]=\emptyset$ for all $j>n$.
 For $U \in \mathcal{S} \cap \JJ_{\mathrm{R}}$ the same statement holds by construction; 
and for $U \in \mathcal{S} \cap \JJ_A$ it holds since $n - 2 \in A$. 
If $\JJ_\II$ would eliminate $[a_n, b_n]$, then $\II$ would eliminate $\uint$, which is impossible.
\qed
\end{proof}

Using Fact~\ref{fact:WWKL}, Theorem~\ref{thm:VCT1} and Propositions~\ref{prop:VCT1-VCT2}, \ref{prop:VCT2-PCR} and \ref{prop:CN-VCT2}
we obtain the following characterization of $\VCT_2$.

\begin{corollary}
\label{cor:VCT2}
$\VCT_2\equivSW\PC_\IR$.
\end{corollary}

We note that the proof of Proposition~\ref{prop:VCT2-PCR} shows that we can extend the domain of 
$\VCT_2$  to sequences $\II$ of intervals that are not almost Vitali covers of $[0,1]$
and Corollary~\ref{cor:VCT2} remains correct for this generalized version of $\VCT_2$.

\section{Countable additivity}
\label{sec:additivity}

In reverse mathematics Brown, Giusto, and Simpson~\cite[Theorem~3.3]{BGS02} (see also Simpson~\cite[Theorem~X.1.9]{Sim09}) have
discussed countable additivity of measures and condition 3.\ of Theorem~\ref{thm:BGS02} turned out to characterize this property.
In this section we would like to analyze this condition in the Weihrauch lattice and
we formulate the condition and a contrapositive version of it in a slightly different way.

\begin{enumerate}
\item Any $\II=(I_n)_n$ that covers $[0,1]$
        satisfies $\sum_{n=0}^\infty\lambda(I_n)\geq1$.
\item For any non-disjoint $\II=(I_n)_n$ that satisfies $\sum_{n=0}^\infty\lambda(I_n)<1$,
        there exists a point $x\in[0,1]\setminus\bigcup\II$.
\end{enumerate}

By a {\em non-disjoint} $\II=(I_n)_n$ we mean one that satisfies $I_i\cap I_j\not=\emptyset$ for some $i\not=j$.
For the correctness of the second statement non-disjointness is not relevant. 
However, it matters for the computational content.
While the first statement has no immediate computational content (more precisely, any reasonable straightforward formalization is computable),
the second one turns out to be equivalent to $*\dash\WWKL$, which we define below.
First we formalize the second statement above as a multi-valued function, which we call the {\em Additive Covering Theorem}.

\begin{definition}[Additive Covering Theorem]\rm
The {\em Additive Covering Theorem} is the multi-valued function $\ACT:\In\Int\mto[0,1], \II\mapsto[0,1]\setminus\bigcup\II$,
where $\dom(\ACT)$ is the set of all non-disjoint $\II=(I_n)_n$ with
$\sum_{n=0}^\infty\lambda(I_n)<1$.
\end{definition}

In order to define $*\dash\WWKL$, we recall that 
for a sequence $f_i:\In X_i\mto Y_i$ we can define the {\em coproduct} $\bigsqcup_{i=0}^\infty f_i:\In\bigsqcup_{i=0}^\infty X_i\mto\bigsqcup_{i=0}^\infty Y_i$,
where $\bigsqcup_{i=0}^\infty Z_i$ denotes the disjoint union of the sets $Z_i$. 
Now we define $*\dash\WWKL:=\bigsqcup_{n=0}^\infty\P_{>2^{-n}}\C_{[0,1]}$, where $\P_{>\varepsilon}\C_{[0,1]}$ is the choice principle
for closed subsets $A\In[0,1]$ with $\lambda(A)>\varepsilon$, as defined in Section~\ref{sec:preliminaries}.
Hence, intuitively, $*\dash\WWKL$ takes as input a number ${n\in\IN}$ together with a closed set $A$ of measure $\lambda(A)>2^{-n}$
and has to produce a point $x\in A$. This could equivalently be defined using quantitative versions of $\WWKL$, hence the name $*\dash\WWKL$
(see Brattka, Gherardi and H\"olzl~\cite[Proposition~7.2]{BGH15a}).
Now we can formulate and prove our main result on $\ACT$.

\begin{theorem}
$\ACT\equivSW*\dash\WWKL$.
\end{theorem}
\begin{proof}
We first prove $\ACT\leqSW*\dash\WWKL$.
Let $\II=(I_n)_n$ be a given non-disjoint sequence of intervals such that $\sum_{n=0}^\infty\lambda(I_n)<1$.
Then we can search for some numbers ${i,j,k\in\IN}$ such that $\varepsilon:=\lambda(I_i\cap I_j)>2^{-k}$.
In this situation we obtain by countable additivity
$\lambda\left(\bigcup_{n=0}^\infty I_n\right)+\varepsilon\leq\sum_{n=0}^\infty\lambda(I_n)<1$.
Hence we obtain for the closed set $A:=[0,1]\setminus\bigcup\II$ that
\[\lambda(A)\geq1-\sum_{n=0}^\infty\lambda(I_n)+\varepsilon>\varepsilon>2^{-k}.\]
Therefore, we can find a point in $A$ using $\P_{>2^{-k}}\C_{[0,1]}(A)$.
This proves the desired reduction $\ACT\leqSW*\dash\WWKL$.

We now prove $*\dash\WWKL\leqSW\ACT$.
Given $k\in\IN$ and a closed set $A\In[0,1]$ such that $\lambda(A)>2^{-k}$ we need to find a point $x\in A$.
The set $A$ is given by a sequence $\JJ$ of open intervals with $A=[0,1]\setminus\bigcup\JJ$.
We can now computably convert the sequence $\JJ$ into a non-disjoint sequence $\II=(I_n)_n$ of open intervals
such that $A=[0,1]\setminus\bigcup\II$ and 
\[\sum_{n=0}^\infty\lambda(I_n)\leq\lambda\left(\bigcup_{n=0}^\infty([0,1]\cap I_n)\right)+2^{-k-1}.\]
This can be achieved if for every $J$ in $\JJ$ we select finitely many intervals $I_n\In J$ such that all intervals selected so far cover $J$  and
such that the overlapping measure of $I_n$ with the
union of the previous intervals (and the exterior of $[0,1]$) is at most $2^{-k-1-n-1}$ for each $n\in\IN$ (and non-zero for at least one $n$).
Since $\lambda(A)>2^{-k}$ we obtain $\lambda\left(\bigcup_{n=0}^\infty([0,1]\cap I_n)\right)<1-2^{-k}$ and the above condition implies
$\sum_{n=0}^\infty\lambda(I_n)<1-2^{-k}+2^{-k-1}<1$
and hence $\ACT(\II)=A$.
This yields the desired reduction.
\qed
\end{proof}

Like $\WWKL\times\C_\IN$ the problem $*\dash\WWKL$ can be seen as a uniform modification of $\WWKL$ that is
indistinguishable from $\WWKL$ when seen from the non-uniform perspective of reverse mathematics.

\begin{figure}[htb]
\begin{tikzpicture}[scale=.5,auto=left,every node/.style={fill=black!15}]

\def\rvdots{\raisebox{1mm}[\height][\depth]{$\huge\vdots$}};


  \node (CR) at (6,16) {$\C_\IR\equivSW\WKL\times\C_\IN$};
   \node (CN) at (26,14) {$\C_\IN$};
  \node (WKL) at (6,14) {$\C_{[0,1]}\equivSW\WKL$};
  \node (WWKLD) at (16,16) {$\VCT_2\equivSW\PC_\IR\equivSW\WWKL\times\C_\IN$};
  \node (WWKL) at (16,14) {$\VCT_1\equivSW\PC_{[0,1]}\equivSW\WWKL$};
  \node (SWWKL) at (16,12) {$\ACT\equivSW*\dash\WWKL$};
  \node (VCT0) at (16,10) {$\VCT_0$};

  \foreach \from/\to in {
  WWKLD/WWKL,
  WWKLD/CN,
  WWKL/SWWKL,
  WKL/WWKL,
  WKL/VCT0,
  CR/WKL,
  CR/WWKLD}
  \draw [->,thick] (\from) -- (\to);

  \foreach \from/\to in {
  SWWKL/VCT0,
  CN/VCT0}
  \draw [->,thick,dashed] (\from) -- (\to);

\end{tikzpicture}
  
\caption{The Vitali Covering Theorem in the Weihrauch lattice. Strong Weihrauch reductions $f\leqSW g$ are indicated by a solid arrow $f\leftarrow g$ and 
similarly ordinary Weihrauch reductions are indicated by a dashed arrow $f\dashleftarrow g$.}
\label{fig:Vitali}
\end{figure}
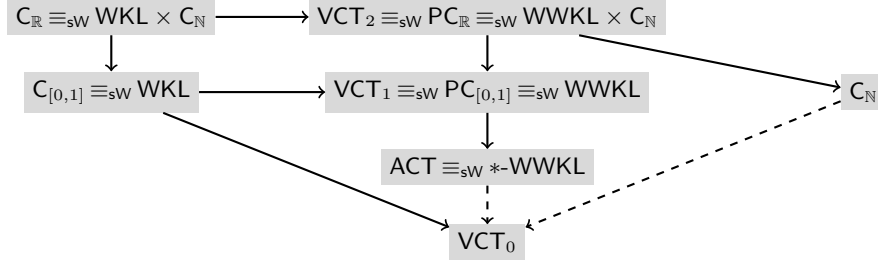

\section{Conclusions}

We have demonstrated that the Vitali Covering Theorem and related results split into several uniform equivalence classes
when analyzed in the Weihrauch lattice. We have summarized the results in the diagram in Figure~\ref{fig:Vitali}. 
The diagram also indicates some equivalence classes in 
the neighborhood that are related to Weak K\H{o}nig's Lemma $\WKL$. 
These classes have not been discussed in this article and some related results can be found in
Brattka, de Brecht and Pauly~\cite{BBP12} and Brattka, Gherardi and H\"olzl~\cite{BGH15a}.

\bibliographystyle{splncs03}
\bibliography{C:/Users/i11avabr/Dropbox/Bibliography/lit}

\end{document}